\RequirePackage{cmap}
\documentclass[preprint,12pt]{elsarticle}

\usepackage[T1]{fontenc}
\usepackage[utf8]{inputenc}
\usepackage{lmodern}
\usepackage{amsmath}
\usepackage{amssymb}

\usepackage{amsthm}
\usepackage{booktabs}
\usepackage{array}
\usepackage{enumitem}
\usepackage{mathtools}
\usepackage{graphicx}

\usepackage[hyphens]{url}
\usepackage[colorlinks=true,citecolor=blue,linkcolor=blue,urlcolor=blue]{hyperref}

\newtheorem{theorem}{Theorem}[section]
\newtheorem{lemma}[theorem]{Lemma}
\newtheorem{proposition}[theorem]{Proposition}

\theoremstyle{definition}
\newtheorem{definition}[theorem]{Definition}
\theoremstyle{remark}

\DeclareMathOperator{\diag}{diag}

\journal{Linear Algebra and its Applications}
\begin{document}

\begin{frontmatter}

\title{Upper bounds for the Laplacian spectral radius:\\Proofs and counterexamples}

\author[inst1,inst2]{Ivan Damnjanovi\'c}
\ead{ivan.damnjanovic@elfak.ni.ac.rs}

\author[inst3]{Taewoo Ha}
\ead{txh2120@gmail.com}

\author[inst4]{Dragan Stevanovi\'c}
\ead{dragan_stevanovic@mi.sanu.ac.rs}

\affiliation[inst1]{
  organization={Faculty of Electronic Engineering, University of Ni\v{s}},
  city={Ni\v{s}},
  country={Serbia}
}

\affiliation[inst2]{
  organization={FAMNIT, University of Primorska},
  city={Koper},
  country={Slovenia}
}

\affiliation[inst3]{
  organization={Independent Researcher},
  city={California},
  country={United States}
}

\affiliation[inst4]{
  organization={Mathematical Institute of the Serbian Academy of Sciences and Arts},
  city={Belgrade},
  country={Serbia}
}

\begin{abstract}
The Laplacian spectral radius of a graph is the largest eigenvalue of its Laplacian matrix. Previously, upper bounds for the Laplacian spectral radius were proposed using a backward-reconstruction procedure starting from expressions equal to $2x$ and substituting local degree data. A numbered list of 68 such candidate bounds was subsequently investigated, resulting in the refutation of 30 of these bounds; two additional bounds were later refuted in a separate study. This paper updates the status of the remaining 36 candidate bounds. Of these remaining bounds, we confirm 22 and refute 12, leaving only two upper bounds open. The valid bounds follow primarily from classical Laplacian spectral radius bounds and the Collatz--Wielandt comparison; the refutations are carried out through explicit counterexamples relying on equitable partitions.
\end{abstract}

\begin{keyword}
Laplacian spectral radius \sep upper bound \sep Collatz--Wielandt comparison \sep
counterexample \sep equitable partition
\MSC[2020] 05C50 \sep 05C07
\end{keyword}

\end{frontmatter}

\section{Introduction}\label{sec:intro}

Let \(G = (V, E)\) be a finite simple graph. The \emph{Laplacian matrix} of $G$ is $L(G) = D(G) - A(G)$, where \(A(G)\) is the adjacency matrix and \(D(G)\) is the diagonal degree matrix. The \emph{Laplacian spectral radius} of $G$, denoted by $\mu(G)$, is the largest eigenvalue of $L(G)$. Upper bounds for \(\mu(G)\) in terms of local degree data are a classical topic in spectral graph theory; see \cite{BrouHae2012}. For any $i \in V$, let $d_i$ denote the degree of vertex $i$. By Geršgorin's theorem, we have $\mu(G) \le 2 \max_{i \in V} d_i$. This result was then improved by Anderson and Morley~\cite{AndMor1985}, who derived the edge-degree bound $\mu(G) \le \max_{ij \in E}(d_i + d_j)$.

Many researchers subsequently took interest in studying upper bounds for $\mu(G)$ involving vertex degrees and average degrees of neighbors. Suppose that $G$ has no isolated vertices, and let $m_i$ denote the average degree of the neighbors of $i \in V$, i.e., $m_i = \frac{1}{d_i} \sum_{j \in N(i)} d_j$, where $N(i)$ denotes the neighborhood of vertex $i$. Among other results, Merris~\cite{Merris1998} obtained the upper bound $\mu(G) \le \max_{i \in V} (d_i + m_i)$, Li and Pan \cite{LiPan2001} showed that $\mu(G) \le \max_{i \in V} \sqrt{2 d_i ( d_i + m_i)}$, while Zhang \cite{Zhang2004} proved the bound $\mu(G) \le \max_{i \in V} \left( d_i + \sqrt{d_i m_i} \right)$. Various other researchers have also contributed to the study of upper bounds on $\mu(G)$ involving $d_i$ and $m_i$, including Das, Guo, Hong, Liu, Lu, Nikiforov, H.~Rojo, O.~Rojo, Shu, Soto, Tian, and Wen-Ren; see \cite{Das2003, Das2004, Guo2005, LiZhang1998, LiuLuTian2004, Nikiforov2007, Pan2002, Rojo2007, RoSoRo2000, ShuHongWenRen2002}.

In a previous paper \cite{BraHanSte2006}, upper bounds for the Laplacian spectral radius were proposed using a backward-reconstruction procedure that starts from expressions equal to $2x$ and substitutes local degree data. With this procedure, a number of candidate upper bounds for $\mu(G)$ were obtained, all of which are of the form
\begin{equation}\label{vertex_type}
    \mu(G) \le \max_{i \in V} f(d_i, m_i),
\end{equation}
where $f$ is a given function, and $G = (V, E)$ is any connected graph on at least two vertices, or
\begin{equation}\label{edge_type}
    \mu(G) \le \max_{ij \in E} f(d_i, m_i, d_j, m_j),
\end{equation}
where $f$ is a given function symmetric with respect to $i$ and $j$, and $G = (V, E)$ is any connected graph on at least two vertices.

In a later study \cite{GheYaKaSte2026_DAM}, a numbered list of 68 of the previously proposed bounds was considered, and 30 of them were refuted using a combined approach of reinforcement learning and an exhaustive search. We mention in passing that the reinforcement learning approach also led to a reimplementation of Wagner's original framework \cite{Wagner2021}; this new framework was later successfully used to obtain other theoretical results \cite{GheYaKaSte2025_AMC, GheYaKaSte2025_ADAM}. Taieb, Roucairol, Cazenave, and Harutyunyan \cite{TaRouCaHa2026} later refuted two more bounds using the Monte Carlo search technique.

Here, we consider the remaining 36 candidate upper bounds on $\mu(G)$ and use the original numbering from \cite{GheYaKaSte2026_DAM} throughout; see Tables \ref{tab:remaining-vertex} and \ref{tab:remaining-edge}. Of these remaining bounds, we confirm 22 and refute 12, leaving only two upper bounds open. Our main result is embodied in the following theorem.

\begin{theorem}\label{thm:main-resolution}
Bounds 1, 4--10, 12, 14, 16, 23, 25--27, 33--35, 37--39, and 42 from Tables \ref{tab:remaining-vertex} and \ref{tab:remaining-edge} are valid, while Bounds 11, 13, 18--22, 24, 30, 40, 47, and 56 are invalid.
\end{theorem}

In light of Theorem \ref{thm:main-resolution}, the only candidate bounds left open are Bounds~44 and~46 from Table \ref{tab:remaining-edge}. Note that Bounds~10 and~23 from Table \ref{tab:remaining-vertex} are identical; hence, we actually consider only 35 mutually distinct candidate upper bounds.

\begin{table}[t]
\centering
\caption{The 24 vertex-type bounds, i.e., bounds of the form \eqref{vertex_type}, among the remaining 36 candidate upper bounds on $\mu(G)$ from \cite{GheYaKaSte2026_DAM}, with the original numbering used. A check mark indicates that the bound is confirmed in this study, while a cross mark indicates that it has been refuted.}
\vspace{4pt}
\label{tab:remaining-vertex}
\small
\setlength{\tabcolsep}{2pt}
\renewcommand{\arraystretch}{1.02}
\begin{tabular*}{\textwidth}{@{\extracolsep{\fill}}rcl@{\quad}rcl@{}}
\toprule
Bound & $f(d_i, m_i)$ & Status & Bound & $f(d_i, m_i)$ & Status \\
\midrule
1 & \(\sqrt{4d_i^3/m_i}\) & \(\checkmark\)
& 16 & \(2d_i^4/m_i^3\) & \(\checkmark\) \\
4 & \(2d_i^2/m_i\) & \(\checkmark\)
& 18 & \(\sqrt{2m_i^3/d_i+2d_i^2}\) & \(\times\) \\
5 & \(d_i^2/m_i+m_i\) & \(\checkmark\)
& 19 & \(\sqrt[4]{4d_i^4 + 12d_i m_i^3}\) & \(\times\) \\
6 & \(\sqrt{m_i^2+3d_i^2}\) & \(\checkmark\)
& 20 & \(\frac12\sqrt{7d_i^2+9m_i^2}\) & \(\times\) \\
7 & \(d_i^2/m_i+d_i\) & \(\checkmark\)
& 21 & \(\sqrt{d_i^3/m_i+3m_i^2}\) & \(\times\) \\
8 & \(\sqrt{d_i(m_i+3d_i)}\) & \(\checkmark\)
& 22 & \(\sqrt[4]{2d_i^4+14d_i^2m_i^2}\) & \(\times\) \\
9 & \((m_i+3d_i)/2\) & \(\checkmark\)
& 23 & \(\sqrt{d_i^2+3d_i m_i}\) & \(\checkmark\) \\
10 & \(\sqrt{d_i(d_i+3m_i)}\) & \(\checkmark\)
& 24 & \(\sqrt[4]{6d_i^4+10m_i^4}\) & \(\times\) \\
11 & \(2m_i^3/d_i^2\) & \(\times\)
& 25 & \(\sqrt[4]{3d_i^4+13d_i^2m_i^2}\) & \(\checkmark\) \\
12 & \(\sqrt{2m_i^2+2d_i^2}\) & \(\checkmark\)
& 26 & \(\frac12\sqrt{5d_i^2+11d_i m_i}\) & \(\checkmark\) \\
13 & \(2m_i^4/d_i^3\) & \(\times\)
& 27 & \(\sqrt{(3d_i^2+5d_i m_i)/2}\) & \(\checkmark\) \\
14 & \(2d_i^3/m_i^2\) & \(\checkmark\)
& 30 & \(m_i^3/d_i^2+d_i^2/m_i\) & \(\times\) \\
\bottomrule
\end{tabular*}
\end{table}

\begin{table}[t]
\centering
\caption{The 12 edge-type bounds, i.e., bounds of the form \eqref{edge_type}, among the remaining 36 candidate upper bounds on $\mu(G)$ from \cite{GheYaKaSte2026_DAM}, with the original numbering used. A check mark indicates that the bound is confirmed in this study, a cross mark indicates that it has been refuted, while \emph{open} means that the candidate bound is left open.}
\vspace{4pt}
\label{tab:remaining-edge}
\small
\setlength{\tabcolsep}{2pt}
\renewcommand{\arraystretch}{1.02}
\begin{tabular*}{\textwidth}{@{\extracolsep{\fill}}rcl@{}}
\toprule
Bound & $f(d_i, m_i, d_j, m_j)$ & Status \\
\midrule
33 & \(2(d_i+d_j)-(m_i+m_j)\) & \(\checkmark\) \\
34 & \(2(d_i^2+d_j^2)/(d_i+d_j)\) & \(\checkmark\) \\
35 & \(2(d_i^2+d_j^2)/(m_i+m_j)\) & \(\checkmark\) \\
37 & \(\sqrt{2(d_i^2+d_j^2)}\) & \(\checkmark\) \\
38 & \(2+\sqrt{2(d_i-1)^2+2(d_j-1)^2}\) & \(\checkmark\) \\
39 & \(2+\sqrt{2(d_i^2+d_j^2)-4(m_i+m_j)+4}\) & \(\checkmark\) \\
40 & \(2+\sqrt{2((m_i-1)^2+(m_j-1)^2)+(d_i^2+d_j^2)-(d_i m_i + d_j m_j)}\) & \(\times\) \\
42 & \(\sqrt{d_i^2+d_j^2+2m_i m_j}\) & \(\checkmark\) \\
44 & \(2+\sqrt{2((d_i-1)^2+(d_j-1)^2+m_i m_j-d_i d_j)}\) & open \\
46 & \(2+\sqrt{2(d_i^2+d_j^2)-16d_i d_j/(m_i+m_j)+4}\) & open \\
47 & \( (2(d_i^2+d_j^2)-(m_i-m_j)^2) / (d_i+d_j)\) & \(\times\) \\
56 & \((d_i^2+d_j^2)(m_i+m_j)/(2d_i d_j)\) & \(\times\) \\
\bottomrule
\end{tabular*}
\end{table}

Methodologically, the confirmations of the valid upper bounds rely on known Laplacian spectral radius inequalities and the Collatz--Wielandt comparison, while the refutations of the invalid bounds are carried out through explicit counterexamples using equitable partitions. To facilitate candidate generation and proof exploration, the authors consulted several AI-assisted systems, namely, AlphaEvolve, AxiomMath, Axplore, ChatGPT, and Claude. Notwithstanding this computational assistance, the authors have rigorously verified every result herein, and all claims rest entirely upon verifiable proofs or counterexamples.

In the rest of the paper, our main focus is to prove Theorem~\ref{thm:main-resolution}. We begin in Section~\ref{sec:preliminaries} by recalling the basic notation, some elementary spectral properties concerning equitable partitions, the Collatz--Wielandt comparison, and well-known upper bounds on the Laplacian spectral radius. Section~\ref{sec:confirmed} then establishes the proofs for the valid upper bounds, while Section~\ref{sec:disproved} provides explicit counterexamples for the invalid bounds through their quotient matrices. Several assertions within Section \ref{sec:confirmed} and all the counterexamples from Section \ref{sec:disproved} are verified computationally using the \texttt{SageMath} \cite{SageMath} scripts available in \cite{GitHub}.

\section{Preliminaries}
\label{sec:preliminaries}

We assume all graphs to be undirected, finite, simple, and without isolated vertices; in particular, we allow graphs to be disconnected. Although the candidate bounds of the form \eqref{vertex_type} or \eqref{edge_type} were originally stated for connected graphs, any disconnected counterexample trivially yields a connected one by considering a component with the largest Laplacian spectral radius.

For a graph $G$, let $V$ and $E$ denote its vertex and edge sets, and let $\Delta$ denote its maximum vertex degree. We let $A$ and $D$ denote the adjacency matrix and the diagonal degree matrix of $G$. The \emph{Laplacian matrix} $L$ and the \emph{signless Laplacian matrix} $Q$ of $G$ are defined as $L = D - A$ and $Q = D + A$. The \emph{Laplacian spectral radius} $\mu(G)$ and the \emph{signless Laplacian spectral radius} $\rho(G)$ are then defined as the largest eigenvalue of $L$ and of $Q$, respectively. For any $i \in V$, let $d_i$ denote the degree of vertex $i$. Furthermore, let $m_i$ denote the average degree of the neighbors of vertex $i$, i.e., $m_i = \frac{1}{d_i} \sum_{j \in N(i)} d_j$, where $N(i)$ denotes the neighborhood of vertex $i$.

For candidate bounds where an expression inside the maximum yields a negative argument under a square root, we adopt the convention that such undefined terms evaluate to $-\infty$, thereby effectively restricting the maximum to those vertex or edge expressions that are well-defined.

\begin{definition}[\hspace{1sp}{\cite[Section 9.3]{GodRoy2001}}]\label{def:equitable}
A partition $V = C_1 \sqcup C_2 \sqcup \cdots \sqcup C_k$ of the vertex set of a graph $G$ is called \emph{equitable} if, for every pair of indices $i, j \in \{ 1, 2, \ldots, k \}$, every vertex in cell $C_i$ has the same number of neighbors in cell $C_j$.
\end{definition}

Let $b_{ij}$ denote the number of neighbors that a vertex in $C_i$ has in $C_j$. The matrix $B = (b_{ij}) \in \mathbb{R}^{k \times k}$ is the \emph{quotient matrix} of the graph $G$ with respect to the equitable partition $C_1 \sqcup C_2 \sqcup \cdots \sqcup C_k$. We then define the \emph{quotient Laplacian matrix} $L_B \in \mathbb{R}^{k \times k}$ as $L_B = \diag(s_1, s_2, \ldots, s_k) - B$, where $s_i = \sum_{j=1}^k b_{ij}$ is the degree of any vertex in $C_i$.

The following lemma establishes the spectral relationship between a graph and its quotient Laplacian matrix, which follows by observing that any eigenvector of $L_B$ can be lifted to an eigenvector of $L$ that is constant on each cell of the partition; see, e.g., \cite[Section 2.3]{BrouHae2012} and \cite[Section 9.3]{GodRoy2001}.

\begin{lemma}\label{lem:quotient}
For any equitable partition of a graph $G$, the spectrum of the corresponding quotient Laplacian matrix $L_B$ is contained in the spectrum of $L$. In particular, $\mu(G)$ is at least the spectral radius of $L_B$.
\end{lemma}

The next lemma provides the necessary and sufficient conditions for a given matrix to be realizable as the quotient matrix of an equitable partition with prescribed cell sizes. These conditions follow directly from standard existence results for regular graphs and semiregular bipartite graphs.

\begin{lemma}\label{lem:realizable}
For some $k \in \mathbb{N}$, let $n_1, n_2, \ldots, n_k \in \mathbb{N}$, and let $B = (b_{ij}) \in \mathbb{R}^{k \times k}$ be a square matrix with nonnegative integer entries. There exists a graph with an equitable partition having cell sizes $n_1, n_2, \ldots, n_k$ and quotient matrix $B$ if and only if:
\begin{enumerate}[label=\textbf{(\arabic*)}]
    \item for every $i \in \{ 1, 2, \ldots, k \}$, $b_{ii} \le n_i - 1$ and at least one of $b_{ii}$ and $n_i$ is even; and
    \item for every $i, j \in \{ 1, 2, \ldots, k \}$ with $i \neq j$, $b_{ij} \le n_j$ and $n_i b_{ij} = n_j b_{ji}$.
\end{enumerate}
\end{lemma}

The next result follows directly from the theory of nonnegative matrices; see \cite[Section~8.7]{GodRoy2001}.

\begin{lemma}\label{lem:LQ}
For every graph $G$, $\mu(G)\le \rho(G)$.
\end{lemma}

In addition, we need the upper bound component of the classical Collatz--Wielandt formula. For a proof of this result, see, e.g., \cite[Section~8.1]{HornJohn2013}.

\begin{lemma}[\hspace{1sp}{\cite[Section~8.1]{HornJohn2013}}]\label{lem:collatz}
Let $M \in \mathbb{R}^{n \times n}$ be a nonnegative matrix, and let $\mathbf{v} \in \mathbb{R}^n$ be a vector with strictly positive entries. If $M \mathbf{v}\le \lambda \mathbf{v}$ entrywise for some $\lambda \in \mathbb{R}$, then the spectral radius of $M$ is at most~$\lambda$.
\end{lemma}

We conclude the section with several well-known upper bounds on the Laplacian spectral radius involving vertex degrees and average degrees of neighbors, the first of which is immediate.

\begin{lemma}\label{lem:two-delta}
For every graph \(G\), $\mu(G) \le 2 \Delta$.
\end{lemma}

\begin{lemma}[\hspace{1sp}{\cite{AndMor1985}}]\label{lem:anderson-morley}
For every graph \(G\), $\mu(G) \le \max_{ij\in E}(d_i+d_j)$.
\end{lemma}

\begin{lemma}[\hspace{1sp}{\cite{Merris1998}}]\label{lem:merris}
For every graph \(G\), $\mu(G)\le \max_{i \in V}(d_i+m_i)$.
\end{lemma}

\begin{lemma}[\hspace{1sp}{\cite{LiPan2001}}]\label{lem:L2-bound}
For every graph \(G\), $\mu(G) \le \max_{i \in V} \sqrt{2 d_i^2 + 2 d_i m_i}$.
\end{lemma}

\section{Confirmed upper bounds}
\label{sec:confirmed}

In this section, we prove the valid upper bounds from Theorem~\ref{thm:main-resolution}. We first address the bounds that follow directly from the classical estimates in Lemmas~\ref{lem:two-delta}--\ref{lem:L2-bound}.

\begin{proposition}\label{valid_prop_1}
    Bounds 1, 4, 5, 7, 14, and 16 from Table \ref{tab:remaining-vertex} are valid.
\end{proposition}
\begin{proof}
    Let $i$ be a vertex of $G$ such that $d_i = \Delta$. Then $m_i \le \Delta$, hence the expressions $\sqrt{4 d_i^3 / m_i}$, $2 d_i^2 / m_i$, $d_i^2 / m_i + d_i$, $2 d_i^3 / m_i^2$, and $2 d_i^4 / m_i^3$ are all at least $2 \Delta$. Also, by the inequality of the arithmetic and geometric means, we have $d_i^2 / m_i + m_i \ge 2 \Delta$. The result now follows from Lemma \ref{lem:two-delta}.
\end{proof}

\begin{proposition}
    Bounds 6 and 9 from Table \ref{tab:remaining-vertex} are valid.
\end{proposition}
\begin{proof}
    Let $i$ be any vertex of $G$. By the inequality of the arithmetic and geometric means, we have $m_i^2 + 3 d_i^2 \ge 2 d_i^2 + 2 d_i m_i$, as well as
    \[
        \left( \frac{m_i + 3d_i}{2} \right)^2 = \frac{9 d_i^2 + 6 d_i m_i + m_i^2}{4} \ge \frac{8 d_i^2 + 8 d_i m_i}{4} = 2 d_i^2 + 2 d_i m_i .
    \]
    Therefore, the expressions $\sqrt{m_i^2 + 3 d_i^2}$ and $(m_i + 3d_i) / 2$ are both at least $\sqrt{2 d_i^2 + 2 d_i m_i}$. The result now follows from Lemma \ref{lem:L2-bound}.
\end{proof}

\begin{proposition}
    Bound 12 from Table \ref{tab:remaining-vertex} is valid.
\end{proposition}
\begin{proof}
    Let $i$ be any vertex of $G$. By the inequality of the arithmetic and geometric means, we have $2 m_i^2 + 2 d_i^2 \ge d_i^2 + 2 d_i m_i + m_i^2 = (d_i + m_i)^2$. Therefore, $\sqrt{2 m_i^2 + 2 d_i^2} \ge d_i + m_i$, so the result follows from Lemma~\ref{lem:merris}.
\end{proof}

\begin{proposition}
    Bounds 33--35 and 37--39 from Table \ref{tab:remaining-edge} are valid.    
\end{proposition}
\begin{proof}
    First, let $ij$ be any edge of $G$. Then, by the inequality of the arithmetic and geometric means, we have $\sqrt{2(d_i^2 + d_j^2)} \ge d_i + d_j$, as well as
    \[
        \frac{2(d_i^2 + d_j^2)}{d_i + d_j} \ge \frac{(d_i + d_j)^2}{d_i + d_j} = d_i + d_j ,
    \]
    and
    \begin{align*}
        2 + \sqrt{2(d_i - 1)^2 + 2(d_j - 1)^2} &= 2 + \sqrt{2d_i^2 + 2d_j^2 - 4d_i - 4d_j + 4}\\
        &\ge 2 + \sqrt{d_i^2 + 2 d_i d_j + d_j^2 - 4d_i - 4d_j + 4}\\
        &= 2 + \sqrt{(d_i + d_j - 2)^2}\\
        &= d_i + d_j .
    \end{align*}
    By Lemma \ref{lem:anderson-morley}, Bounds 34, 37, and 38 from Table \ref{tab:remaining-edge} are valid.

    Now, let $ij$ be an edge of $G$ maximizing $d_i + d_j$. Then each neighbor of $i$ has degree at most $d_j$, and each neighbor of $j$ has degree at most $d_i$. Hence $m_i \le d_j$ and $m_j \le d_i$. Therefore, $2(d_i + d_j) - (m_i + m_j) \ge d_i + d_j$, as well as $2(d_i^2 + d_j^2) / (m_i + m_j) \ge 2(d_i^2 + d_j^2) / (d_i + d_j)$, and
    \begin{align*}
        2 + \sqrt{2(d_i^2 + d_j^2) - 4(m_i + m_j) + 4} &\ge 2 + \sqrt{2(d_i^2 + d_j^2) - 4(d_i + d_j) + 4}\\
        &= 2 + \sqrt{2(d_i - 1)^2 + 2(d_j - 1)^2} .
    \end{align*}
    The validity of Bounds 33, 35, and 39 follows from Lemma \ref{lem:anderson-morley} and the already established validity of Bounds 34 and 38.
\end{proof}

We now turn to Bound 42 from Table \ref{tab:remaining-edge}, which we confirm using a different idea.

\begin{proposition}\label{valid_prop_5}
    Bound 42 from Table \ref{tab:remaining-edge} is valid.
\end{proposition}
\begin{proof}
By Lemma \ref{lem:LQ}, it suffices to bound $\rho(G)$. Suppose that $G$ is connected, and let $P \in \mathbb{R}^{V \times V}$ be the diagonal matrix with $P_{ii} = d_i$ for every $i \in V$. Then $B = P^{-1} Q P$, which satisfies
\[
    B_{ij} = \begin{cases}
        d_i, & \text{if $j = i$},\\
        d_j / d_i, & \text{if $j \neq i$ and $j \in N(i)$,}\\
        0, & \text{otherwise},
    \end{cases}
\]
for any $i, j \in V$, is similar to $Q$.

Let $\mathbf{v}$ be a positive Perron vector of $B$. Choose $i$ with maximal $\mathbf{v}_i$, then choose $j \in N(i)$ with maximal $\mathbf{v}_j$ among the neighbors of $i$. The Perron vector equations give $(\rho(G) - d_i) \mathbf{v}_i \le m_i \, \mathbf{v}_j$ and $(\rho(G) - d_j) \mathbf{v}_j \le m_j \, \mathbf{v}_i$. Since $\rho(G)$ is at least any diagonal entry of $B$, we have $\rho(G) \ge d_i, d_j$, hence $(\rho(G) - d_i)(\rho(G) - d_j) \le m_i m_j$. Thus, solving the quadratic inequality gives
\[
    \rho(G) \le \frac{d_i + d_j + \sqrt{(d_i - d_j)^2 + 4m_i m_j}}{2} .
\]

To finalize the proof, it is enough to show that
\begin{equation}\label{aux_2}
    \frac{d_i + d_j + \sqrt{(d_i - d_j)^2 + 4m_i m_j}}{2} \le \sqrt{d_i^2 + d_j^2 + 2m_i m_j} .
\end{equation}
By squaring both sides and simplifying, \eqref{aux_2} comes down to proving
\begin{equation}\label{aux_5}
    (d_i + d_j) \sqrt{(d_i - d_j)^2 + 4 m_i m_j} \le d_i^2 + d_j^2 + 2m_i m_j .
\end{equation}
By squaring and simplifying again, \eqref{aux_5} reduces to $d_i^2 d_j^2 + m_i^2 m_j^2 \ge 2 d_i d_j m_i m_j$, which trivially holds.
\end{proof}

In the remainder of the section, we confirm the remaining valid upper bounds from Theorem \ref{thm:main-resolution} via the Collatz--Wielandt comparison, making use of the following technical lemma.

\begin{lemma}\label{lem:cw}
    Let $G$ be a graph, $\lambda > \Delta$ be a real number, and $\varphi \colon (0, 1) \to (0, +\infty)$ be a concave function. Suppose further that for every $i \in V$, we have $d_i \, \varphi(\frac{m_i}{\lambda}) \le (\lambda - d_i) \, \varphi(\frac{d_i}{\lambda})$. Then $\mu(G) \le \lambda$.
\end{lemma}
\begin{proof}
    Since $\lambda > \Delta$, note that $\frac{d_i}{\lambda}, \frac{m_i}{\lambda} \in (0, 1)$ for every $i \in V$. Consider the vector $\mathbf{v} \in \mathbb{R}^V$ defined by $\mathbf{v}_i = \varphi(\frac{d_i}{\lambda})$. Then, for every $i \in V$, $(Q \mathbf{v})_i = d_i \, \varphi(\frac{d_i}{\lambda}) + \sum_{j \in N(i)} \varphi(\frac{d_j}{\lambda})$. Since $\varphi$ is concave, by Jensen's inequality, we have $\sum_{j \in N(i)} \varphi(\frac{d_j}{\lambda}) \le d_i \, \varphi(\frac{m_i}{\lambda})$. Therefore,
    \[
        \frac{(Q \mathbf{v})_i}{\mathbf{v}_i} \le d_i + \frac{d_i}{\varphi(\frac{d_i}{\lambda})} \varphi(\tfrac{m_i}{\lambda}) \le d_i + (\lambda - d_i) = \lambda .
    \]
    By Lemma \ref{lem:collatz}, we have $\rho(G) \le \lambda$, whence Lemma \ref{lem:LQ} yields $\mu(G) \le \lambda$.
\end{proof}

We now apply Lemma~\ref{lem:cw} to confirm the remaining 
valid upper bounds, selecting an appropriate $\lambda > \Delta$ 
in each instance.

\begin{proposition}\label{valid_prop_7}
    Bound 8 from Table \ref{tab:remaining-vertex} is valid.
\end{proposition}
\begin{proof}
    Consider the concave function $\varphi(x) = x^{1/4}$ on $(0, 1)$. By invoking Lemma \ref{lem:cw} with $\lambda = \max_{i \in V} \sqrt{d_i(m_i + 3d_i)}$, it suffices to prove $d_i (\frac{m_i}{\lambda})^{1 / 4} \le (\lambda - d_i) (\frac{d_i}{\lambda})^{1 / 4}$ for every $i \in V$, which is equivalent to $1 + (\frac{m_i}{d_i})^{1/4} \le \frac{\lambda}{d_i}$. Let $i$ be any vertex of $G$. Since $\lambda \ge \sqrt{d_i(m_i + 3d_i)}$, it is enough to show that $1 + r^{1 / 4} \le \sqrt{3 + r}$, with $r = \frac{m_i}{d_i}$. The result follows by observing that
    \[
        (3 + r) - (1 + r^{1 / 4})^2 = (r^\frac{1}{4} - 1)^2 (r^{1/2} + 2 r^{1/4} + 2) \ge 0
    \]
    for any $r > 0$.
\end{proof}

\begin{proposition}
    Bound 27 from Table \ref{tab:remaining-vertex} is valid.
\end{proposition}
\begin{proof}
    Consider the concave function $\varphi(x) = x^{5/8}$ on $(0, 1)$. By invoking Lemma \ref{lem:cw} with $\lambda = \max_{i \in V} \sqrt{(3d_i^2 + 5 d_i m_i) / 2}$, it suffices to prove $d_i (\frac{m_i}{\lambda})^{5/8} \le (\lambda - d_i) (\frac{d_i}{\lambda})^{5 / 8}$ for every $i \in V$, which is equivalent to $1 + (\frac{m_i}{d_i})^{5/8} \le \frac{\lambda}{d_i}$.
    
    Let $i$ be any vertex of $G$. Since $i$ has a neighbor $j$ such that $d_j \ge m_i$, and $\lambda \ge \sqrt{(3 d_j^2 + 5 d_j m_j) / 2} > d_j \sqrt{3 / 2}$, it follows that $\lambda > m_i \sqrt{3 / 2}$. Let $r = \frac{m_i}{d_i}$, and observe that $\frac{\lambda}{d_i} > r \sqrt{3 / 2}$. Also, since $\lambda \ge \sqrt{(3d_i^2 + 5 d_i m_i) / 2}$, we have $\frac{\lambda}{d_i} \ge \sqrt{(3 + 5r) / 2}$. Therefore, to complete the proof, it suffices to show that
    \[
        1 + r^{5 / 8} \le \max \left\{ \sqrt{(3 + 5r) / 2}, r \sqrt{3 / 2} \right\}
    \]
    for any $r > 0$. This can be rigorously confirmed by executing the script \texttt{bound\_confirmation.py} from~\cite{GitHub}, which relies on the built-in \texttt{SageMath} real root isolation algorithm.
\end{proof}

\begin{proposition}
    Bounds 10 and 23 from Table \ref{tab:remaining-vertex}, which are identical, are valid.
\end{proposition}
\begin{proof}
    Consider the concave function
    \[
        \varphi(x) = \begin{cases}
            x, & \text{if $x \in (0, \frac{1}{2}]$},\\
            \frac{1 + x}{3}, & \text{if $x \in (\frac{1}{2}, 1)$},
        \end{cases}
    \]
    on $(0, 1)$. By invoking Lemma \ref{lem:cw} with $\lambda = \max_{i \in V} \sqrt{d_i^2 + 3 d_i m_i}$, it suffices to prove
    \begin{equation}\label{aux_3}
        d_i \, \varphi(\tfrac{m_i}{\lambda}) \le (\lambda - d_i) \, \varphi(\tfrac{d_i}{\lambda})
    \end{equation}
    for every $i \in V$. Let $i$ be any vertex of $G$. Depending on whether $\frac{d_i}{\lambda} \le \frac{1}{2}$, we split the argument into two cases.

    \bigskip\noindent
    \textbf{Case 1.} $\frac{d_i}{\lambda} \le \frac{1}{2}$.

    In this case, $\varphi(\frac{d_i}{\lambda}) = \frac{d_i}{\lambda}$, so \eqref{aux_3} transforms to $\varphi(\frac{m_i}{\lambda}) \le 1 - \frac{d_i}{\lambda}$. If $\frac{m_i}{\lambda} \le \frac{1}{2}$, then $\varphi(\frac{m_i}{\lambda}) = \frac{m_i}{\lambda} \le \frac{1}{2} \le 1 - \frac{d_i}{\lambda}$, and the result follows. Now, suppose that $\frac{m_i}{\lambda} > \frac{1}{2}$. If $\frac{d_i}{\lambda} \le \frac{1}{3}$, we have $\varphi(\frac{m_i}{\lambda}) = \frac{1 + \frac{m_i}{\lambda}}{3} < \frac{2}{3} \le 1 - \frac{d_i}{\lambda}$, so the result follows once again.

    It remains to settle the subcase where $\frac{m_i}{\lambda} > \frac{1}{2}$ and $\frac{d_i}{\lambda} \in (\frac{1}{3}, \frac{1}{2}]$. In this situation, $(2 d_i - \lambda)(4 d_i - \lambda) \le 0$, which implies $\lambda^2 - d_i^2 \le 3 d_i (2 \lambda - 3 d_i)$, i.e.,
    \begin{equation}\label{aux_4}
        \frac{\lambda^2 - d_i^2}{3 d_i} \le 2 \lambda - 3 d_i .
    \end{equation}
    Since $\lambda \ge \sqrt{d_i^2 + 3 d_i m_i}$, we have $m_i \le \frac{\lambda^2 - d_i^2}{3 d_i}$, which together with \eqref{aux_4} gives $m_i \le 2 \lambda - 3 d_i$. Therefore, $\varphi(\frac{m_i}{\lambda}) = \frac{1 + \frac{m_i}{\lambda}}{3} \le 1 - \frac{d_i}{\lambda}$.

    \bigskip\noindent
    \textbf{Case 2.} $\frac{d_i}{\lambda} > \frac{1}{2}$.

    In this case, $\varphi(\frac{d_i}{\lambda}) = \frac{\lambda + d_i}{3 \lambda}$, so \eqref{aux_3} transforms to $d_i \, \varphi(\frac{m_i}{\lambda}) \le \frac{\lambda^2 - d_i^2}{3 \lambda}$. Regardless of whether $\frac{m_i}{\lambda} \le \frac{1}{2}$, we have $\varphi(\frac{m_i}{\lambda}) \le \frac{m_i}{\lambda}$, so it suffices to prove that $\frac{d_i m_i}{\lambda} \le \frac{\lambda^2 - d_i^2}{3 \lambda}$. The result now follows directly from $\lambda \ge \sqrt{d_i^2 + 3 d_i m_i}$.
\end{proof}

\begin{proposition}
    Bound 25 from Table \ref{tab:remaining-vertex} is valid.
\end{proposition}
\begin{proof}
    Consider the concave function $\varphi(x) = x(1 - x)^{3/16}$ on $(0, 1)$. By invoking Lemma \ref{lem:cw} with $\lambda = \max_{i \in V} \sqrt[4]{3d_i^4+13d_i^2m_i^2}$, it suffices to prove $d_i \frac{m_i}{\lambda} (1 - \frac{m_i}{\lambda})^{3/16} \le (\lambda - d_i) \frac{d_i}{\lambda}(1 - \frac{d_i}{\lambda})^{3/16}$ for every $i \in V$, which is equivalent to $(\frac{m_i}{\lambda})^{16/19} (1 - \frac{m_i}{\lambda})^{3/19} \le 1 - \frac{d_i}{\lambda}$.

    Let $i$ be any vertex of $G$. Since $\lambda \ge \sqrt[4]{3d_i^4 + 13d_i^2 m_i^2}$, we have $3(\frac{d_i}{\lambda})^4 + 13 (\frac{d_i}{\lambda})^2 (\frac{m_i}{\lambda})^2 - 1 \le 0$, so solving the quadratic inequality for $(\frac{d_i}{\lambda})^2$ yields
    \[
        \frac{d_i}{\lambda} \le \sqrt{ \frac{-13(\frac{m_i}{\lambda})^2 + \sqrt{169(\frac{m_i}{\lambda})^4 + 12}}{6} } .
    \]
    Therefore, to complete the proof, it is enough to show that
    \begin{equation}\label{aux_6}
        \left(\frac{m_i}{\lambda}\right)^{16/19} \left(1 - \frac{m_i}{\lambda} \right)^{3/19} \le 1 - \sqrt{ \frac{-13(\frac{m_i}{\lambda})^2 + \sqrt{169(\frac{m_i}{\lambda})^4 + 12}}{6} } .
    \end{equation}
    Put $\beta = (\frac{\lambda - m_i}{m_i})^{1/19}$, so that $\beta > 0$ and $\frac{m_i}{\lambda} = \frac{1}{1 + \beta^{19}}$. Then \eqref{aux_6} reduces to
    \begin{equation}\label{aux_7}
        \sqrt{\frac{-\frac{13}{(1 + \beta^{19})^2} + \sqrt{\frac{169}{(1 + \beta^{19})^4} + 12}}{6}} \le 1 - \frac{\beta^3}{1 + \beta^{19}} . 
    \end{equation}
    Observe that the right-hand side of \eqref{aux_7} is positive because $(1 + \beta^{19}) - \beta^3 = \beta^3 (\beta^{16} - 1) + 1 > 0$. By squaring both sides and simplifying, \eqref{aux_7} comes down to proving
    \begin{equation}\label{aux_8}
        \sqrt{\frac{169}{(\beta^{19} + 1)^4} + 12} \le \frac{6(\beta^{19} - \beta^3 + 1)^2 + 13}{(\beta^{19} + 1)^2} .
    \end{equation}
    By squaring and simplifying again, \eqref{aux_8} reduces to
    \[
        \left( 6(\beta^{19} - \beta^3 + 1)^2 + 13 \right)^2 - 169 \ge 12 (\beta^{19} + 1)^4 ,
    \]
    which holds for every $\beta > 0$, as rigorously confirmed by the \texttt{SageMath} script \texttt{bound\_confirmation.py} from~\cite{GitHub}.
\end{proof}

\begin{proposition}\label{valid_prop_11}
    Bound 26 from Table \ref{tab:remaining-vertex} is valid.
\end{proposition}
\begin{proof}
    Consider the concave function $\varphi(x) = x(1 - x)^{5/16}$ on $(0, 1)$. By invoking Lemma \ref{lem:cw} with $\lambda = \max_{i \in V} \frac{1}{2} \sqrt{5d_i^2+11d_i m_i}$, it suffices to prove $d_i \frac{m_i}{\lambda} (1 - \frac{m_i}{\lambda})^{5/16} \le (\lambda - d_i) \frac{d_i}{\lambda}(1 - \frac{d_i}{\lambda})^{5/16}$ for every $i \in V$, which is equivalent to $(\frac{m_i}{\lambda})^{16/21} (1 - \frac{m_i}{\lambda})^{5/21} \le 1 - \frac{d_i}{\lambda}$.

    Let $i$ be any vertex of $G$. Since $\lambda \ge \frac{1}{2} \sqrt{5d_i^2 + 11d_i m_i}$, we have $5(\frac{d_i}{\lambda})^2 + 11 \frac{d_i}{\lambda} \frac{m_i}{\lambda} - 4 \le 0$, so solving the quadratic inequality for $\frac{d_i}{\lambda}$ yields
    \[
        \frac{d_i}{\lambda} \le \frac{-11\frac{m_i}{\lambda} + \sqrt{121(\frac{m_i}{\lambda})^2 + 80}}{10}.
    \]
    Therefore, to complete the proof, it is enough to show that
    \begin{equation}\label{aux_9}
        \left(\frac{m_i}{\lambda}\right)^{16/21} \left(1 - \frac{m_i}{\lambda} \right)^{5/21} \le 1 - \frac{-11\frac{m_i}{\lambda} + \sqrt{121(\frac{m_i}{\lambda})^2 + 80}}{10} .
    \end{equation}
    Put $\beta = (\frac{\lambda - m_i}{m_i})^{1/21}$, so that $\beta > 0$ and $\frac{m_i}{\lambda} = \frac{1}{1 + \beta^{21}}$. Then \eqref{aux_9} reduces to
    \[
        \frac{-\frac{11}{1 + \beta^{21}} + \sqrt{\frac{121}{(1 + \beta^{21})^2} + 80}}{10} \le 1 - \frac{\beta^5}{1 + \beta^{21}} ,
    \]
    i.e.,
    \begin{equation}\label{aux_10}
        \sqrt{\frac{121}{(\beta^{21} + 1)^2} + 80} \le \frac{10(\beta^{21} - \beta^5 + 1) + 11}{\beta^{21} + 1} .
    \end{equation}
    Observe that the right-hand side of \eqref{aux_10} is positive because $\beta^{21} - \beta^5 + 1 = \beta^5 (\beta^{16} - 1) + 1 > 0$. By squaring both sides and simplifying, \eqref{aux_10} comes down to proving
    \[
        \left( 10(\beta^{21} - \beta^5 + 1) + 11 \right)^2 - 121 \ge 80(\beta^{21} + 1)^2 ,
    \]
    which indeed holds for every $\beta > 0$, as confirmed by the \texttt{SageMath} script \texttt{bound\_confirmation.py} from~\cite{GitHub}.
\end{proof}

The affirmative part of Theorem \ref{thm:main-resolution} now follows from Propositions \ref{valid_prop_1}--\ref{valid_prop_5} and \ref{valid_prop_7}--\ref{valid_prop_11}. We mention in passing that, as suggested by some of the proofs, all of these valid upper bounds are also upper bounds for $\rho(G)$.
Moreover, any expression of the form $\max_{i \in V} f(d_i, m_i)$ for some function $f$, or $\max_{ij \in E} f(d_i, m_i, d_j, m_j)$ for some function $f$ symmetric with respect to $i$ and $j$, is an upper bound for $\mu(G)$ for all graphs $G$ if and only if it is an upper bound for $\rho(G)$ for all graphs $G$.

Indeed, if a graph is a counterexample for the Laplacian spectral radius bound, then by Lemma \ref{lem:LQ}, it is also a counterexample for the signless Laplacian spectral radius bound. On the other hand, if a graph is a counterexample for the signless Laplacian spectral radius bound, then it is straightforward to verify that its Kronecker cover~\cite{Pisanski2018} is a counterexample for the Laplacian spectral radius bound.

\section{Refuted upper bounds}
\label{sec:disproved}

In the present section, we provide a constructive and computer-assisted proof of the negative part of Theorem \ref{thm:main-resolution}, which is embodied in the following proposition.

\begin{proposition}\label{refutation_prop}
    Bounds 11, 13, 18--22, 24, and 30 from Table \ref{tab:remaining-vertex}, and Bounds 40, 47, and 56 from Table \ref{tab:remaining-edge} are invalid.
\end{proposition}

\begin{table}[t]
\centering
\caption{The counterexample graphs that refute each of the bounds from Proposition \ref{refutation_prop}, described via their equitable partitions with specified cell sizes and associated quotient matrices.}
\vspace{4pt}
\label{tab:counterexamples}
\small
\setlength{\tabcolsep}{2pt}
\renewcommand{\arraystretch}{1.02}
\begin{tabular*}{\textwidth}{@{\extracolsep{\fill}}rcl@{\quad}rcl@{}}
\toprule
Bound & Cell sizes & Quotient & Bound & Cell sizes & Quotient \\
\midrule
11 & $(1, 7, 21)$ & $\begin{bmatrix}\begin{smallmatrix}
0 & 7 & 0\\
1 & 0 & 3\\
0 & 1 & 2
\end{smallmatrix}\end{bmatrix}$
& 22 & $(7, 9, 14)$ & $\begin{bmatrix}\begin{smallmatrix}
0 & 9 & 2\\
7 & 0 & 0\\
1 & 0 & 0
\end{smallmatrix}\end{bmatrix}$\\[5pt]
13 & $(1, 9, 36)$ & $\begin{bmatrix}\begin{smallmatrix}
0 & 9 & 0\\
1 & 0 & 4\\
0 & 1 & 3
\end{smallmatrix}\end{bmatrix}$
& 24 & $(69, 19, 76)$ & $\begin{bmatrix}\begin{smallmatrix}
0 & 0 & 76\\
0 & 0 & 8\\
69 & 2 & 0
\end{smallmatrix}\end{bmatrix}$\\[5pt]
18 & $(48, 12, 42, 56)$ & $\begin{bmatrix}\begin{smallmatrix}
0 & 1 & 0 & 42\\
4 & 0 & 14 & 0\\
0 & 4 & 0 & 4\\
36 & 0 & 3 & 0
\end{smallmatrix}\end{bmatrix}$
& 30 & $(83, 38, 44, 88)$ & $\begin{bmatrix}\begin{smallmatrix}
0 & 0 & 0 & 88\\
0 & 0 & 22 & 0\\
0 & 19 & 0 & 6\\
83 & 0 & 3 & 0
\end{smallmatrix}\end{bmatrix}$\\[8pt]
19 & $(5, 5, 5, 5)$ & $\begin{bmatrix}\begin{smallmatrix}
0 & 5 & 1 & 0\\
5 & 0 & 0 & 1\\
1 & 0 & 0 & 0\\
0 & 1 & 0 & 0
\end{smallmatrix}\end{bmatrix}$
& 40 & $(1, 39, 312)$ & $\begin{bmatrix}\begin{smallmatrix}
0 & 39 & 0\\
1 & 0 & 8\\
0 & 1 & 0
\end{smallmatrix}\end{bmatrix}$\\[8pt]
20 & $(5, 5, 5, 5)$ & $\begin{bmatrix}\begin{smallmatrix}
0 & 5 & 1 & 0\\
5 & 0 & 0 & 1\\
1 & 0 & 0 & 0\\
0 & 1 & 0 & 0
\end{smallmatrix}\end{bmatrix}$
& 47 & $(28, 2, 22, 7)$ & $\begin{bmatrix}\begin{smallmatrix}
0 & 2 & 0 & 1\\
28 & 0 & 0 & 0\\
0 & 0 & 0 & 7\\
4 & 0 & 22 & 0
\end{smallmatrix}\end{bmatrix}$\\[8pt]
21 & $(69, 19, 76)$ & $\begin{bmatrix}\begin{smallmatrix}
0 & 0 & 76\\
0 & 0 & 8\\
69 & 2 & 0
\end{smallmatrix}\end{bmatrix}$
& 56 & $(1, 5, 15, 15)$ & $\begin{bmatrix}\begin{smallmatrix}
0 & 5 & 0 & 0\\
1 & 0 & 3 & 0\\
0 & 1 & 0 & 1\\
0 & 0 & 1 & 0
\end{smallmatrix}\end{bmatrix}$\\
\bottomrule
\end{tabular*}
\end{table}

Proposition~\ref{refutation_prop} is proved constructively by 
exhibiting a specific counterexample for each invalid bound. These graphs are described in Table~\ref{tab:counterexamples} via their equitable partitions with specified cell sizes and associated quotient matrices. The refutation follows a uniform three-step procedure for each candidate bound. First, 
Lemma~\ref{lem:realizable} is applied to guarantee the structural realizability of the given cell sizes and quotient matrix, ensuring the existence of such a graph~$G$. Next, the quotient Laplacian matrix is computed, and Lemma~\ref{lem:quotient} is used to determine 
a lower bound on $\mu(G)$. Finally, this value 
is verified to strictly exceed the respective 
right-hand side expressions from~\eqref{vertex_type} 
or~\eqref{edge_type}, thereby confirming that $G$ 
is a valid counterexample.

This refutation process can be carried out in a rigorous, computer-assisted manner using the \texttt{SageMath} script \texttt{bound\_refutation.py} available in~\cite{GitHub}. To ensure absolute mathematical certainty and eliminate any potential issues with floating-point rounding errors, the script 
utilizes symbolic computation and a built-in real root isolation algorithm. This guarantees that the eigenvalue estimations are sufficiently precise to definitively confirm each refutation.

\section*{Declarations}
\noindent\emph{Generative AI and AI-assisted technologies.} The authors have used the AI-assisted systems AlphaEvolve, AxiomMath, Axplore, ChatGPT, and Claude for counterexample search, algebraic exploration, code generation, proof discussion, and drafting. The authors reviewed all resulting material; correctness rests on the provided proofs and exact counterexamples.
\smallskip

\noindent\emph{Funding.} I.~Damnjanović is supported by the Ministry of Science, Technological Development and Innovation of the Republic of Serbia, grant number 451-03-34/2026-03/200102, and the Science Fund of the Republic of Serbia, grant \#6767, Lazy walk counts and spectral radius of threshold graphs --- LZWK.
\smallskip

\noindent\emph{Competing interests.} The authors declare no competing interests.
\smallskip

\noindent\emph{Data and code availability.} No experimental dataset was used. The \texttt{SageMath} scripts used to computationally verify several assertions and all the counterexamples are available in \cite{GitHub}.

\bibliographystyle{amcjoucc}
\bibliography{references}

@article{AndMor1985,
  author  = {Anderson, Jr., W. N. and Morley, T. D.},
  title   = {{Eigenvalues of the Laplacian of a graph}},
  journal = {Linear Multilinear Algebra},
  volume  = {18},
  number  = {2},
  pages   = {141--145},
  year    = {1985},
  url     = {http://www.doi.org/10.1080/03081088508817681}
}

@article{BraHanSte2006,
  author  = {Brankov, Vladimir and Hansen, Pierre and Stevanovi\'c, Dragan},
  title   = {{Automated conjectures on upper bounds for the largest Laplacian eigenvalue of graphs}},
  journal = {Linear Algebra Appl.},
  volume  = {414},
  pages   = {407--424},
  year    = {2006},
  url     = {http://www.doi.org/10.1016/j.laa.2005.10.017}
}

@book{BrouHae2012,
  author    = {Brouwer, A. E. and Haemers, W. H.},
  title     = {Spectra of Graphs},
  series    = {Universitext},
  publisher = {Springer},
  address   = {New York, NY, USA},
  year      = {2012},
  url       = {http://www.doi.org/10.1007/978-1-4614-1939-6}
}

@misc{GitHub,
    author = {Damnjanović, Ivan and Ha, Taewoo and Stevanović, Dragan},
    title = {{Upper bounds for the Laplacian spectral radius: Proofs and counterexamples: Supplementary material (GitHub repository)}},
    url = {https://github.com/Ivan-Damnjanovic/bhs-bounds}
}

@article{Das2003,
  author  = {Das, Kinkar C.},
  title   = {{An improved upper bound for Laplacian graph eigenvalues}},
  journal = {Linear Algebra Appl.},
  volume  = {368},
  pages   = {269--278},
  year    = {2003},
  url     = {http://www.doi.org/10.1016/S0024-3795(02)00687-0}
}

@article{Das2004,
  author  = {Das, Kinkar Ch.},
  title   = {{The Laplacian spectrum of a graph}},
  journal = {Comput. Math. Appl.},
  volume  = {48},
  number  = {5--6},
  pages   = {715--724},
  year    = {2004},
  url     = {http://www.doi.org/10.1016/j.camwa.2004.05.005}
}

@article{GheYaKaSte2025_AMC,
  author  = {Ghebleh, Mohammad and Al-Yakoob, Salem and Kanso, Ali and Stevanovi\'c, Dragan},
  title = {Graphs having two main eigenvalues and arbitrarily many distinct vertex degrees},
  journal = {Appl. Math. Comput.},
  volume = {495},
  year = {2025},
  pages = {129311},
  url = {http://www.doi.org/10.1016/j.amc.2025.129311}
}

@article{GheYaKaSte2025_ADAM,
  author  = {Ghebleh, Mohammad and Al-Yakoob, Salem and Kanso, Ali and Stevanovi\'c, Dragan},
  title = {{Reinforcement learning for graph theory, II. Small Ramsey numbers}},
  journal = {Art Discrete Appl. Math.},
  volume = {8},
  year = {2025},
  pages = {\#P1.07},
  url = {http://www.doi.org/10.26493/2590-9770.1788.8af}
}

@article{GheYaKaSte2026_DAM,
  author  = {Ghebleh, Mohammad and Al-Yakoob, Salem and Kanso, Ali and Stevanovi\'c, Dragan},
  title   = {{Reinforcement learning for graph theory, I: Reimplementation of Wagner's approach}},
  journal = {Discrete Appl. Math.},
  volume  = {380},
  pages   = {468--479},
  year    = {2026},
  url     = {http://www.doi.org/10.1016/j.dam.2025.10.047}
}

@book{GodRoy2001,
  author    = {Godsil, Chris and Royle, Gordon},
  title     = {Algebraic Graph Theory},
  series    = {Graduate Texts in Mathematics},
  volume    = {207},
  publisher = {Springer},
  address   = {New York, NY, USA},
  year      = {2001},
  url       = {http://www.doi.org/10.1007/978-1-4613-0163-9}
}

@article{Guo2005,
  author  = {Guo, Ji-Ming},
  title   = {{A new upper bound for the Laplacian spectral radius of graphs}},
  journal = {Linear Algebra Appl.},
  volume  = {400},
  pages   = {61--66},
  year    = {2005},
  url     = {http://www.doi.org/10.1016/j.laa.2004.10.022}
}

@book{HornJohn2013,
  author    = {Horn, Roger A. and Johnson, Charles R.},
  title     = {Matrix Analysis},
  edition   = {2},
  publisher = {Cambridge University Press},
  address   = {Cambridge, UK},
  year      = {2013},
  url = {https://doi.org/10.1017/CBO9780511810817},
}

@article{LiPan2001,
  author  = {Li, Jiong-Sheng and Pan, Yong-Liang},
  title   = {{De Caen's inequality and bounds on the largest Laplacian eigenvalue of a graph}},
  journal = {Linear Algebra Appl.},
  volume  = {328},
  pages   = {153--160},
  year    = {2001},
  url     = {http://www.doi.org/10.1016/S0024-3795(00)00307-4}
}

@article{LiZhang1998,
  author  = {Li, Jiong-Sheng and Zhang, Xiao-Dong},
  title   = {{On the Laplacian eigenvalues of a graph}},
  journal = {Linear Algebra Appl.},
  volume  = {285},
  pages   = {305--307},
  year    = {1998},
  url     = {http://www.doi.org/10.1016/S0024-3795(98)10149-0}
}

@article{LiuLuTian2004,
  author  = {Liu, Huiqing and Lu, Mei and Tian, Feng},
  title   = {{On the Laplacian spectral radius of a graph}},
  journal = {Linear Algebra Appl.},
  volume   = {376},
  pages    = {135--141},
  year     = {2004},
  url      = {https://doi.org/10.1016/j.laa.2003.06.007}
}

@article{Merris1998,
  author  = {Merris, Russell},
  title   = {{A note on Laplacian graph eigenvalues}},
  journal = {Linear Algebra Appl.},
  volume  = {285},
  number  = {1--3},
  pages   = {33--35},
  year    = {1998},
  url     = {http://www.doi.org/10.1016/S0024-3795(98)10148-9}
}

@article{Nikiforov2007,
  author  = {Nikiforov, Vladimir},
  title   = {{Bounds on graph eigenvalues~II}},
  journal = {Linear Algebra Appl.},
  volume  = {427},
  pages   = {183--189},
  year    = {2007},
  url     = {http://www.doi.org/10.1016/j.laa.2007.07.010}
}

@article{Pan2002,
  author  = {Pan, Yong-Liang},
  title   = {{Sharp upper bounds for the Laplacian graph eigenvalues}},
  journal = {Linear Algebra Appl.},
  volume  = {355},
  number  = {1--3},
  pages   = {287--295},
  year    = {2002},
  url     = {http://www.doi.org/10.1016/S0024-3795(02)00353-1}
}

@article{Pisanski2018,
  author  = {Pisanski, Toma{\v{z}}},
  title   = {Not every bipartite double cover is canonical},
  journal = {Bull. Inst. Comb. Appl.},
  volume  = {82},
  year    = {2018},
  pages   = {51--55},
  url     = {https://bica.the-ica.org/Volumes/82//Reprints/BICA2017-27-Main-Reprint.pdf}
}

@article{Rojo2007,
  author  = {Rojo, Oscar},
  title   = {{A nontrivial upper bound on the largest Laplacian eigenvalue of weighted graphs}},
  journal = {Linear Algebra Appl.},
  volume  = {420},
  pages   = {625--633},
  year    = {2007},
  url     = {http://www.doi.org/10.1016/j.laa.2006.08.022}
}

@article{RoSoRo2000,
    author = {Rojo, Oscar and Soto, Ricardo and Rojo, Héctor},
    title = {{An always nontrivial upper bound for Laplacian graph eigenvalues}},
    journal = {Linear Algebra Appl.},
    volume = {312},
    year = {2000},
    pages = {155--159},
    url = {http://www.doi.org/10.1016/S0024-3795(00)00104-X}
}

@article{ShuHongWenRen2002,
  author  = {Shu, Jin-Long and Hong, Yuan and Wen-Ren, Kai},
  title   = {{A sharp upper bound on the largest eigenvalue of the Laplacian matrix of a graph}},
  journal = {Linear Algebra Appl.},
  volume  = {347},
  pages   = {123--129},
  year    = {2002},
  url     = {http://www.doi.org/10.1016/S0024-3795(01)00548-1}
}

@inproceedings{TaRouCaHa2026,
  author    = {Taieb, Liora and Roucairol, Milo and Cazenave, Tristan and Harutyunyan, Ararat},
  title     = {{Automated refutation with Monte Carlo search of graph theory conjectures on the maximum Laplacian eigenvalue}},
  editor = {Zhang, Yingqian and Hladík, Milan and Moosaei, Hossein},
  booktitle = {Learning and Intelligent Optimization},
  series = {Lecture Notes in Computer Science},
  volume = {15744},
  publisher = {Springer},
  address = {Cham, Switzerland},
  year      = {2026},
  pages     = {52--63},
  url       = {http://www.doi.org/10.1007/978-3-032-09156-7_4}
}

@misc{Wagner2021,
  author        = {Wagner, Adam Zsolt},
  title         = {Constructions in combinatorics via neural networks},
  year          = {2021},
  note          = {\href{https://arxiv.org/abs/2104.14516}
                  {\texttt{arXiv:2104.14516}}},
}

@article{Zhang2004,
  author  = {Zhang, Xiao-Dong},
  title   = {{Two sharp upper bounds for the Laplacian eigenvalues}},
  journal = {Linear Algebra Appl.},
  volume  = {376},
  pages   = {207--213},
  year    = {2004},
  url     = {http://www.doi.org/10.1016/S0024-3795(03)00644-X}
}

@misc{SageMath,
  author = {{The Sage Developers}},
  title = {{SageMath, the Sage Mathematics Software System (Version 10.8)}},
  year = {2025},
  url = {https://www.sagemath.org}
}

\end{document}